\documentclass[12pt]{amsart}
\usepackage[left=2.5cm,right=2.5cm, top=2.5cm, bottom=2.5cm]{geometry}
\usepackage{caption,subcaption, tikz, color, amsfonts, mathtools, amsmath,mathrsfs, amsthm, stmaryrd, amssymb}
\usepackage{hyperref}
\usepackage{verbatim}
\usetikzlibrary{patterns}
\usepackage{graphicx}
\usepackage{enumerate}

\title{Two-player incentive compatible outcome functions are affine maximizers}
\author{Bo Lin} 
\address[Bo Lin]{Department of Mathematics, University of Texas at Austin, Texas TX 78712, USA}
\email{linbomath@gmail.com}
\author{Ngoc Mai Tran}
\address[Ngoc Mai Tran]{Department of Mathematics, University of Texas at Austin, Texas TX 78712, USA and Hausdorff Center for Mathematics, Bonn 53115, Germany}
\email{ntran@math.utexas.edu}
\date{\today}

\DeclareMathOperator{\R}{\mathbb{R}}

\DeclareMathOperator\tdet{tdet}

\theoremstyle{plain}
\newtheorem{theorem}{Theorem}[section]
\newtheorem*{theorem*}{Theorem}
\newtheorem{lemma}[theorem]{Lemma}
\newtheorem{proposition}[theorem]{Proposition}
\newtheorem{corollary}[theorem]{Corollary}
\theoremstyle{definition}
\newtheorem{definition}[theorem]{Definition}
\newtheorem{example}[theorem]{Example}

\newtheorem{remark}[theorem]{Remark}

\newcommand{\minH}{\underline{\mathcal{H}}}

\newcommand{\calC}{\mathcal{C}}
\newcommand{\calS}{\mathcal{S}}
\newcommand{\calT}{\mathcal{T}}
\newcommand{\T}{\mathcal{T}}

\DeclareMathOperator*{\IC}{IC}
\DeclareMathOperator*{\AM}{AM}


\begin{document}
	
\begin{abstract}
	In mechanism design, for a given type space, there may be incentive compatible outcome functions which are not affine maximizers. Using tools from linear algebra and tropical geometry, we prove that for two-player games on a discrete type space, any given outcome function can be turned into an affine maximizer through a nontrivial perturbation of the type space. Furthermore, our theorems are the strongest possible in this setup. 
\end{abstract}

\maketitle

\section{Introduction}
Mechanisms are engineered games, devised to implement social choice functions in an economy where the players are rational, act in their own interests, and make decisions based on private information called their \emph{type}. A special case is single-item auction, where the players are bidders, a player's type is what they truly think the good on sales is worth, and a desirable social choice function could be the one that would assign the good to the person who values it the most, for example. In general, mechanism design encompasses much broader setups. The theory is rich and fundamental to game theory, economics and computer science \cite{laffont2009theory,vohra2011mechanism}. However, even in the simplest case, some fundamental questions remain open. This paper concerns one such problem, namely, how far is the set of incentive compatible outcomes from the set of affine maximizers in a discrete type space. 

Throughout this paper, we consider the simplest setup of deterministic, dominant strategy incentive compatible (IC) mechanisms on $m$ outcomes and $n$ players. This means the following. The game can take one of $m \geq 1$ (finitely many) possible outcomes. There are $n \geq 1$ players, each has a true type $t^i \in \R^m$, known only to them, where $t^i_j$ is how much player $i$ values outcome $j$. The mechanism designer does not see the true types. Instead, she sees the individual type spaces $\calT^1, \dots, \calT^n$, where each $\calT^i$ is an ordered tuple of points in $\R^m$ that contains the true type of the $i$-th player.  The designer needs to choose a mechanism $(g,p)$, which consists of an outcome function $g: \calT := \calT^1 \times \dots \times \calT^n \to [m]$, together with a payment function $p: \calT \to \R^n$. Given a mechanism $(g,p)$, the game works as follows. Each player declares a type (which may or may not be their true type), and based on that declared type $s = (s^1, \dots, s^n)$ in $\calT$, the game outcome is $g(s)$, and player $i \in [n]$ needs to pay $p(s)_{i}$ to the designer. When the players declare $s$ in $\calT$, the profit of player $i$ is $t^i_{g(s)} - p(s)_{i}$. 
The players are assumed to be rational, which means that given other players' declared types, player $i$ will choose to declare $u$ in $\calT^{i}$ that maximizes her profit. This $u$ may or may not be her true type. The mechanism is incentive compatible in dominant strategy, abbreviated (IC), if: for every player $i$, declaring her true type would maximize her profit, regardless of what the others declare. That is, for each fixed $s = (s^1, \dots, s^{i-1}, t^i, s^{i+1}, \dots s^n)$, for all $u \in \calT^i$,
\begin{equation}\label{eqn:ic}
 t^i_{g(s)} - p(s)_{i} \geq t^i_{g(s^1, \dots, s^{i-1}, u, s^{i+1}, \dots, s^n)} - p(s^1, \dots, s^{i-1}, u, s^{i+1}, \dots, s^n)_{i}. 
\end{equation}
Incentive compatible mechanisms are simple and strategy proof, and thus encourage a more transparent and efficient economy.
 
A fundamental problem in mechanism design is to characterize all possible IC mechanisms on a given type space $\calT$. For single-player mechanisms (where $n = 1$), there are multiple characterizations \cite{vohra2011mechanism}. For example, it is the set of all covectors corresponding to cells in the tropical polytope of $\calT$ \cite{CrowellTran1}. However, when $n \geq 2$, the problem remains wide open, except for a handful of special cases \cite{bartal2003incentive,lavi2003towards}. One particular subset of multi-player IC outcome functions heavily studied is the set of affine maximizers. Affine maximizers are tractable precisely because they can be handled using techniques from the single-player case. However, generally they form a strict subset of all possible IC outcome functions on $\calT$. There are very few spaces where all IC outcome functions are affine maximizers. The first result in this direction is Roberts' Theorem \cite{roberts1979characterization}, which states that if $\calT^i = \R^m$ for all $i$, then all IC outcome functions are affine maximizers. Since then, much efforts have been put in understanding and extending Roberts' Theorem, often to rather specific type spaces \cite{carbajal2013truthful,dobzinski2009modular,edelman2017dominant,lavi2009two, mishra2012roberts,vohra2011mechanism}, see \cite{mishra2018separability} for a recent overview of this literature. 

Our paper takes a novel view on Roberts' Theorem and its successors. Instead of looking for type spaces where all IC outcome functions are affine maximizers, we seek to quantify how large is the latter set relative to the former. Specifically, we ask: is a given IC outcome function $g$ on an $n$-player type space $\calT = \calT^1 \times \dots \times \calT^n$ an affine maximizer on some type space $\calS = \calS^1 \times \dots \times \calS^n$, where $\IC(\calT)$ and $\IC(\calS)$, the sets of IC outcome functions on $\calT$ and $\calS$ respectively, differ as little as possible?
In other words, can one perturb the type space slightly to include a given IC outcome function~$g$ in its set of affine maximizers? Call this the \emph{perturbed Roberts' problem}. 

Type space perturbation has appeared in the context of robust mechanism design
\cite{bergemann2005robust,bergemann2012robust,chung2007foundations, oury2012continuous}. 
Though the context is different, we share a common purpose with this literature, namely, to produce theorems in mechanism design that hold for even when the underlying type space is slightly disturbed through numerical roundings or arbitrary noise perturbations. Such results are especially important for practical applications such as auctions, where the player's true valuation of the object can only be approximated to start with. Another compelling reason to measure the gap between affine maximizers and all IC outcome functions rests in the main challenge of mechanism design: to find mechanisms which are optimal according to some criterion, for instance, manipulation-free or maximizes revenue. In these cases, one would like to know whether optimizing over all affine maximizers is very far from optimizing over all IC outcome functions \cite{dobzinski2011limitations,lavi2003towards}. In this view, the perturbed Roberts' problem is a measure of the approximability of IC outcome functions by affine maximizers. 

The exact result on the perturbed Roberts' problem depends on how much $\IC(\calS)$ is allowed to differ from $\IC(\calT)$. Under the tropical geometry view of mechanism design, for the single-player case, $\IC(\calT)$ is a topological constraint on the tropical hyperplane arrangements formed by $\calT$ \cite{horn2016topological}. Clearly approximation is easier but less interesting when the topological constraint is less rigid. Thus finding the strongest constraint where approximation still holds is particularly interesting. The main result of our paper, Theorem \ref{thm:main} is such a positive approximation result. Theorems \ref{thm:main.counter} and \ref{thm:2p-sym.counter} show that approximability dramatically fails under various relaxations of the constraints in Theorem \ref{thm:main}, thus proving that the result we obtained is the strongest possible in these settings.

\begin{theorem}[Two-player, one fixed type, one outcome function]\label{thm:main}
Let $\calT = \calT^1 \times \dots \times \calT^{n}$ be a finite type space on $n$ players and $m \geq 1$ outcomes. Write $\calS = \calS^1 \times \dots \times \calS^{n}$. Under the constraints $n = 2$ and $\calS^1 = \calT^1$, for any $g \in \IC(\calT^1 \times \calT^2)$, there exists a type space $\calS^2$ such that $g$ is an affine maximizer on $\calS = \calT^1 \times \calS^2$.
\end{theorem}

\begin{theorem}\label{thm:main.counter}
None of the constraints in Theorem \ref{thm:main} can be relaxed. In other words, there exist type spaces $\calT$ where the following versions of the perturbed Roberts' problem have no solution. 
\begin{enumerate}
	\item[(i)] For $n \geq 3$, find $\calS$ such that $g$ is an affine maximizer on $\calS$ and $\IC(\calS^1) = \IC(\calT^1)$.
	\item[(ii)] For $n \geq 2$, find $\calS$ such that $g$ is an affine maximizer on $\calS$ and $\IC(\calS) = \IC(\calT)$.
	\item[(iii)] For $n \geq 2$, find $\calS$ such that for any pair $g,h \in \IC(\calT)$, $g$ and $h$ are both affine maximizers on $\calS$ and $\IC(\calS^1) = \IC(\calT^1)$.
\end{enumerate}
\end{theorem}
We remark that the constraints imposed in Theorem \ref{thm:main} are non-trivial. The given outcome function $g$ is determined by~$\calT$, so $g$ being an affine maximizer on $\calS$ means that $\calS^2$ must satisfy \emph{some} of the inequalities and equalities as those satisfied by $\calT^2$. Thus $\calS^2$ cannot be trivial.

The three cases in Theorem \ref{thm:main.counter} correspond to various generalizations of Theorem \ref{thm:main}. Case (i) allows three or more players and $\calS^1$ to be different from $\calT^1$, as long as they have the same set of IC outcome functions. Case (ii) requires that the set of IC outcome functions in the second player must also match up, that is, $\IC(\calS^2) = \IC(\calT^2)$. Case (iii) requires that the same type space can be chosen for two different IC outcome functions $g$ and $h$. Theorem \ref{thm:main.counter} shows that without further constraints, all of these possible extensions of Theorem \ref{thm:main} do not hold. Thus, Theorem \ref{thm:main} is the strongest possible in this setup.

Our second major result explores neutral outcome functions in two-player games with identical type spaces, that is, $\calT^1 = \calT^2$ and $g$ is symmetric. This setup is particularly interesting in applications, for it means that the mechanism respects the player's anonymity. Theorem \ref{thm:main} implies that one can choose a type space $\calS = \calS^1 \times \calS^2$ such that $g$ is an affine maximizer on $\calS$ and $\IC(\calS^{1}) = \IC(\calT^{1})$.  One may ask if one can further require $\calS^1=\calS^2$. The answer is affirmative only if one drops the assumption that $\IC(\calS^{1}) = \IC(\calT^{1})$. 
\begin{theorem}\label{thm:2p-sym.counter}
There exists $\calT^1$ and symmetric $g \in \IC(\calT^1,\calT^1)$ such that for any $\calS^1\in \R^{r\times m}$ with $\IC(\calS^1)=\IC(\calT^1)$, we have that $g$ is not an affine maximizer on $\calS = \calS^1 \times \calS^1$.
\end{theorem}

Our proof technique builds on a theorem of tropical geometry that relates $\IC(\calT^1)$ to the tropical determinant of the minors of the matrix $\calT^1$. In economics term, this theorem states that revealing the set of IC outcomes on an individual type space (without specifying the type space itself) is equivalent to revealing the player's preference orders when only a subset of the type space and outcomes is available. This theorem is the matching field characterization of tropical hyperplane arrangements \cite[Theorem 4.4, Example 4.6]{fink2015stiefel}, \cite{JoswigLoho}, which dates back to classical results in discrete geometry \cite{babson1998geometry,gelfand2008discriminants,sturmfels1993maximal}. In our paper, it plays a crucial role in establishing the equalities between the set of IC outcome functions in $\calS^1$ and $\calT^1$ in Theorem \ref{thm:main}. Furthermore, it gives intuition on the intricate construction used to prove Theorem~\ref{thm:main.counter}. 

Our paper is organized as follows. In Section \ref{sec:background}, we give a brief introduction to mechanism design via tropical geometry and provide some examples to illustrate Theorem \ref{thm:main}. This approach was first described in \cite{CrowellTran1}. For a comprehensive introduction to either topic, we recommend \cite{vohra2011mechanism} for mechanism design, and \cite{maclagan2015introduction} for tropical geometry. 
In Section \ref{sec:typespaces}, we present the definitions and tools needed for the proofs. We prove Theorems \ref{thm:main} and \ref{thm:main.counter} in Section \ref{sec:proof}, and conclude with discussions and open questions in Section \ref{sec:summary}.
\vskip12pt \noindent
\textbf{Notation.} For an integer $m$, let $[m] = \{1,2,\dots,m\}$. Fixing $m$, let $\mathbf{0}$ be the zero vector in $\R^{m}$. For an outcome function $g: \calT := \calT^1 \times \dots \times \calT^n \to [m]$, since each $\calT^{i}$ is ordered, we also denote $g$ by a $r_{1}\times r_{2}\times \cdots \times r_{n}$-tensor with entries in $[m]$. In particular, $g$ becomes a vector in $[m]^{r_{1}}$ when $n=1$, and $g$ becomes a matrix in $[m]^{r_{1}\times r_{2}}$ when $n=2$. Given type spaces $\calT^{i}$ with type matrices $T^{i}\in \R^{r_{i}\times m}$, and for constants $\alpha_{1},\alpha_{2},\ldots,\alpha_{n}>0$ and a vector $\gamma\in \R^{m}$, we define another type space $M(\alpha_{1}\calT^{1}, \ldots, \alpha_{n}\calT^{n}, \gamma)$ as follows: it is an ordered tuple with $\prod_{i=1}^{n}{r_{i}}$ entries, where the $(i_{1},i_{2},\ldots,i_{n})$-th entry is the vector  
\begin{equation}\label{eqn:tensorcomb}
	\alpha_{1}\calT^{1}_{i_{1}} + \cdots + \alpha_{n}\calT^{n}_{i_{n}} + \gamma \in \R^{m},
\end{equation}
for $1\le i_{j}\le r_{j}, j=1,2,\ldots,n$, and the entries are ordered lexicographically.

\section{Background} \label{sec:background}

\subsection{Background on mechanism design}
First we collect relevant notations and definitions from mechanism design. 
Consider a game with $m \in \mathbb{N}$ outcomes and $n \in \mathbb{N}$ players. For $i = 1, \dots, n$, the type space of a single player $i$ is an ordered tuple $\calT^i$ with $r_{i}$ entries where each entry is a vector in $\R^m$. The multi-player type space $\calT$ is the Cartesian product $\calT := \calT^1 \times \dots \times \calT^n$. A mechanism $(g,p)$ on the type space $\calT$ consists of a function $g: \calT \to [m]$ and a payment function $p: \calT \to \R^n$. Suppose player $i$ has true type $t^i$ in $\calT^i$. The mechanism $(g,p)$ is incentive compatible (IC) if for all $i$, \eqref{eqn:ic} holds. An outcome function $g$ is incentive compatible if there exists a payment function $p$ such that $(g,p)$ is IC. An outcome function $g$ is called an affine maximizer on $\calT$ if there exist weights $\alpha_1, \dots, \alpha_n > 0$, a constant vector $\gamma \in \R^m$, and a single-player IC outcome function $h: M(\alpha_{1}\calT^{1}, \ldots, \alpha_{n}\calT^{n}, \gamma) \to [m]$ such that $g(t^1, \dots, t^n) = h(\alpha_{1}t^{1} + \cdots + \alpha_{n}t^{n} + \gamma)$ for all $(t^{1},t^{2},\ldots,t^{n}) \in \calT$.
Let $\IC(\calT^1, \dots, \calT^n) \subseteq [m]^{\prod_{i=1}^{n}{r_{i}}}$ denote the set of all IC outcome functions on $\calT$, and $\AM(\calT^1, \dots, \calT^n)$ denote the set of all affine maximizers on $\calT$. The type matrix of $\calT$ is the matrix $T \in \R^{r \times m}$, where the row vectors of $T$ are the entries of $\calT$ listed in the same order as $\calT$.

In a single player game $(n = 1)$, we shall drop the superscript $1$ in all notations. In this case, a classical result in mechanism design \cite[\S 3]{vohra2011mechanism} states that $(g,p)$ is IC on $\calT$ if and only if there exists a function $q: [m] \to \R^{n}$ such that
$$ p = q \circ g. $$
Following the convention of this literature, in this case, we call $q$ the payment function and denote it by $p$. Then (\ref{eqn:ic}) becomes
\begin{equation}\label{eqn:ic.2}
 t_{g(t)} - p(g(t)) \geq t_{g(s)} - p(g(s))
\end{equation}
for all pairs $s,t \in \calT$. Note that $g$ is defined on the $r$ entries in $\calT$ and its values are in $[m]$, so for convenience we also write $g$ as a vector in $[m]^{r}$. 

\subsection{Mechanism design via tropical geometry}

The tropical min-plus algebra $(\mathbb{R}\cup \{\infty\}, \underline{\oplus},\odot)$ is defined by $a \underline{\oplus} b := \min(a,b), a \odot b := a + b$. For a point $p \in \R^m$, $j \in [m]$, the \emph{min-plus tropical hyperplane} with apex $p$, denoted $\underline{\mathcal{H}}(p)$, is the set of $z \in \R^m$ such that the minimum in the tropical inner product
\begin{equation}\label{eqn:pz.min}
(-p)^\top \underline{\odot} z = \min\{z_1 - p_1, \ldots, z_m - p_m\}
\end{equation}
is achieved at least twice. We can similarly define the tropical max-plus algebra by replacing $\infty$ with $-\infty$; $\underline{\oplus}$ with $\overline{\oplus}$; $\min(a,b)$ with $\max(a,b)$. 
For $J \subseteq [m]$, the $J$ \emph{closed sector} $\minH_J(p)$ of $\minH(p)$ is the set of $z \in \R^m$ such that the minimum in (\ref{eqn:pz.min}) is achieved at all $j \in J$. When $J=\{j\}$ is a singleton, we also write $\minH_{j}$ for $\minH_{J}$. The collection of non-empty sectors $(\minH_J(p) \neq \emptyset: J \subseteq [m])$ partitions $\R^m$ into a classical polyhedral complex. The \emph{tropical hyperplane arrangement} based on $\calT$, denoted $\minH(\calT)$ is the intersection of the polyhedral complexes formed by the closed sectors of the tropical hyperplanes $(\minH(t^i) : i=1,\ldots,r)$. A vector $h\in [m]^{r}$ is called a \emph{covector} of $\minH(\calT)$ if there exists $z\in \R^{m}$ such that for $1\le i\le r$, 
\[ z\in \minH_{h(i)}(t^i). \]
In other words, $j=h(i)$ attains the minimum of $z_{j}-t^{i}_{j}$ among $1\le j\le m$. The key connection between mechanism design and tropical geometry is the following \cite[Theorem 1]{CrowellTran1}. 

\begin{proposition}\label{prop:sector}
The set of covectors of $\minH(\calT)$ is $\IC(\calT)$.
\end{proposition}

We exploit Proposition \ref{prop:sector} here as a powerful visualization technique. Since one can visualize $\minH(\calT)$ for $m = 3$ in $\R^2$, this gives an easy way to list out all IC outcome functions on an individual type space $\calT$, as demonstrated in Example \ref{exa:pos} below.

\begin{example}[Visualization of IC outcome functions with tropical hyperplane arrangements]\label{exa:arrangement}
	Let $r_{1}=r_{2}=3$ and $m=3$. Define
	\begin{equation*}
	T^{1}=\begin{bmatrix}
	0& 2& 3 \\
	0& 4& 2 \\
	0& 3& 7
	\end{bmatrix}, \quad T^{2}=\begin{bmatrix}
	0& 5& -5 \\
	0& -2& 9 \\
	0& 1& 3
	\end{bmatrix}.
	\end{equation*}
Here $T^1$ is the type matrix of the type space $\mathcal{T}^1 = \left((0,2,3), (0,4,2), (0,3,7)\right) \in \left(\R^{3}\right)^{3}$. Similarly, $T^2$ is the type matrix of $\mathcal{T}^2 = \left((0,5,-5), (0,-2,9), (0,1,3)\right) \in \left(\R^{3}\right)^{3}$. Under the map $(x_0,x_1,x_2) \in \R^3 \mapsto (x_1-x_0,x_2-x_0) \in \R^2$, Figure \ref{fig:arrangement} shows the tropical hyperplane $\minH(t^1_3) = \minH(0,3,7)$ together with the three closed sectors $\minH_1(t^1_3), \minH_2(t^1_3)$ and $\minH_3(t^1_3)$, shaded as a visual guide. Figure \ref{fig:arrangement}(b) and (c) shows the tropical hyperplane arrangements $\minH(\calT^{1})$ and $\minH(\calT^{2})$ and covectors corresponding to the open, full-dimensional cells, unshaded for simplicity. 

Now consider a mechanism $g: \mathcal{T}^1 \times \mathcal{T}^2 \to [3]$, represented as the following matrix
$$ g =\begin{bmatrix}
	2& 1& 1 \\
	2& 1& 2 \\
	3& 3& 3
	\end{bmatrix}.$$
This means $g(t^1_1, t^2_3)$ is the $(1,3)$ entry of the matrix, which is $1$, for example. One can check, using Figure \ref{fig:arrangement}, that the columns of $g$ are covectors of $\calT^1$. This means for each $t^2 \in \mathcal{T}^2$, $g(\cdot, t^2): \mathcal{T}^1 \to [3]$ is IC. Similarly, the rows of $g$ are covectors of $\calT^2$. This means for each $t^1 \in \mathcal{T}^1$, $g(t^1, \cdot): \mathcal{T}^2 \to [3]$ is IC. Therefore, by definition, $g$ is a two-player IC outcome function on $\mathcal{T}^1 \times \mathcal{T}^2$.

	\begin{figure}[ht!]
		\begin{minipage}[t]{0.3\textwidth}
			\centering
			\begin{tikzpicture}[scale=0.53]
			\draw[help lines, black!15] (0,0) grid (9,15);
  
  		\filldraw [black!10, fill opacity=0.7] (3.05,7.05) -- (3.05,14.95) -- (8.95,14.95) -- (8.95,7.05);			

  		\filldraw [black!20, fill opacity=0.7] (2.95,6.95) -- (2.95,14.95) -- (0.05,14.95) -- (0.05,4.05) -- (2.95,6.95);			  		
  		
			\filldraw[red] (3,7) circle (2.5pt);
			\draw [red, line width=0.1mm] (0,4) -- (3,7) -- (3,15);
			\draw [red, line width=0.1mm] (3,7) -- (9,7);
			\node[below right] at (3,7) {\color{red} $t^{1}_{3}$};
			\draw[line width=0.1mm] (0,0) rectangle (9,15);  		
			
			\node at (6,11) {\color{red} $\minH_1(t^1_3)$};
			\node at (1.5,9) {\color{red} $\minH_2(t^1_3)$};			
			\node at (5,3) {\color{red} $\minH_3(t^1_3)$};						
			
			\end{tikzpicture}
		\end{minipage}	
		\begin{minipage}[t]{0.33\textwidth}
			\centering
			\begin{tikzpicture}[scale=0.53]
			\draw[help lines, black!15] (0,0) grid (9,15);			
			\node at (7,4) {\small $({\color{blue} 1}{\color{brown} 1}{\color{red} 3})$};
			\node at (3,4) {\small $({\color{blue} 1}{\color{brown} 2}{\color{red} 3})$};
			\node at (1,3.5) {\small $({\color{blue} 2}{\color{brown} 2}{\color{red} 3})$};
			\node[below] at (2,2) {\small $({\color{blue} 3}{\color{brown} 2}{\color{red} 3})$};			
			\node at (7,1) {\small $({\color{blue} 3}{\color{brown} 3}{\color{red} 3})$};						
			\node at (7,2.5) {\small $({\color{blue} 3}{\color{brown} 1}{\color{red} 3})$};									
			\node at (7,8) {\small $({\color{blue} 1}{\color{brown} 1}{\color{red} 1})$};			
			\node at (3.75,9) {\small $({\color{blue} 1}{\color{brown} 2}{\color{red} 1})$};						
			\node at (2.75,11) {\small $({\color{blue} 1}{\color{brown} 2}{\color{red} 2})$};									
			\node at (1,8) {\small $({\color{blue} 2}{\color{brown} 2}{\color{red} 2})$};												
			
			\filldraw[blue] (2,3) circle (2.5pt);
			\draw [blue, line width=0.1mm] (0,1) -- (2,3) -- (2,15);
			\draw [blue, line width=0.1mm] (2,3) -- (9,3);
			\node[below right] at (2,3) {\color{blue} $t^{1}_{1}$};
			
			\filldraw[brown] (4,2) circle (2.5pt);
			\draw [brown, line width=0.1mm] (2,0) -- (4,2) -- (4,15);
			\draw [brown, line width=0.1mm] (4,2) -- (9,2);
			\node[below right] at (4,2) {\color{brown} $t^{1}_{2}$};
			
			\filldraw[red] (3,7) circle (2.5pt);
			\draw [red, line width=0.1mm] (0,4) -- (3,7) -- (3,15);
			\draw [red, line width=0.1mm] (3,7) -- (9,7);
			\node[below right] at (3,7) {\color{red} $t^{1}_{3}$};
			
			\draw[line width=0.1mm] (0,0) rectangle (9,15);
			
			\end{tikzpicture}
		\end{minipage}
		\begin{minipage}[t]{0.33\textwidth}
			\centering
			\begin{tikzpicture}[scale=0.4]
				\draw[help lines, black!15] (-4,-8) grid (8,12);			
			\node at (3,11) {\small $({\color{blue} 2}{\color{brown} 1}{\color{red} 1})$};
			\node at (-0.5,11) {\small $({\color{blue} 2}{\color{brown} 1}{\color{red} 2})$};
			\node at (5,-7) {\small $({\color{blue} 3}{\color{brown} 3}{\color{red} 3})$};

			\node at (0,-2) {\small $({\color{blue} 2}{\color{brown} 3}{\color{red} 3})$};
			\node at (-2,4) {\small $({\color{blue} 2}{\color{brown} 3}{\color{red} 2})$};			
			\node[below] at (-2.5,10.5) {\small $({\color{blue} 2}{\color{brown} 2}{\color{red} 2})$};			
			\node at (3,6) {\small $({\color{blue} 2}{\color{brown} 3}{\color{red} 1})$};			
			\node at (6.5,6) {\small $({\color{blue} 1}{\color{brown} 3}{\color{red} 1})$};						
			\node at (6.5,-2) {\small $({\color{blue} 1}{\color{brown} 3}{\color{red} 3})$};			
			\node at (6.5,11) {\small $({\color{blue} 1}{\color{brown} 1}{\color{red} 1})$};									
			
			\filldraw[blue] (5,-5) circle (2.5pt);
			\draw [blue, line width=0.1mm] (2,-8) -- (5,-5) -- (5,12);
			\draw [blue, line width=0.1mm] (5,-5) -- (8,-5);
			\node[above left] at (5,-5) {\color{blue} $t^{2}_{1}$};
			
			\filldraw[brown] (-2,9) circle (2.5pt);
			\draw [brown, line width=0.1mm] (-4,7) -- (-2,9) -- (-2,12);
			\draw [brown, line width=0.1mm] (-2,9) -- (8,9);
			\node[below right] at (-2,9) {\color{brown} $t^{2}_{2}$};
			
			\filldraw[red] (1,3) circle (2.5pt);
			\draw [red, line width=0.1mm] (-4,-2) -- (1,3) -- (1,12);
			\draw [red, line width=0.1mm] (1,3) -- (8,3);
			\node[above left] at (1,3) {\color{red} $t^{2}_{3}$};
			
			\draw[line width=0.1mm] (-4,-8) rectangle (8,12);
			
			\end{tikzpicture}
		\end{minipage}
			\caption{Figures accompany Example \ref{exa:arrangement}. From left to right: the tropical hyperplane at $t_3^1$ and its three closed sectors $\minH_1(t^1_3), \minH_2(t^1_3)$ and $\minH_3(t^1_3)$; the tropical hyperplane arrangement $\minH(\calT^{1})$ and the tropical hyperplane arrangement $\minH(\calT^{2})$, respectively, together with the covectors corresponding to their full-dimensional cells. Figures are drawn to scale, but the coordinates are not shown for simplicity.}
		\label{fig:arrangement}
	\end{figure}
\end{example}

\begin{example}[A demonstration of Theorem \ref{thm:main}]\label{exa:pos}
Continue from Example \ref{exa:arrangement}, we saw that $g \in \IC(\mathcal{T}^1,\mathcal{T}^2)$. However, we claim that $g \notin \AM(\mathcal{T}^1,\mathcal{T}^2)$. Suppose there exist $\alpha_{1},\alpha_{2}>0$ and $\gamma \in \R$ such that $g\in \IC(M(\alpha_{1}\mathcal{T}^{1},\alpha_{2}\mathcal{T}^{2},\gamma))$. Then by Lemma \ref{lem:am}, there exists $z\in \R^{3}$ such that for $1\le i,j,k\le 3$ and $k\ne g(i,j)$,

\begin{equation*}
	\alpha_{1}T^{1}_{i,g(i,j)}+\alpha_{2}T^{2}_{j,g(i,j)}+\gamma-z_{g(i,j)} \ge \alpha_{1}T^{1}_{i,k}+\alpha_{2}T^{2}_{j,k}+\gamma-z_{k}.
\end{equation*}

So we have
\vskip-1cm
\begin{align}
	z_{1}-z_{2} &\ge -2\alpha_{1}-5\alpha_{2} \quad (i=1,j=1,g(i,j)=2,k=1); \label{eqn:1-z111} \\
	z_{3}-z_{1} &\ge 3\alpha_{1}+9\alpha_{2} \quad (i=1,j=2,g(i,j)=1,k=3); \label{eqn:1-z123} \\
	z_{2}-z_{1} &\ge 4\alpha_{1}-2\alpha_{2} \quad (i=2,j=2,g(i,j)=1,k=2); \label{eqn:1-z222} \\
	z_{2}-z_{3} &\ge -4\alpha_{1}+10\alpha_{2} \quad (i=3,j=1,g(i,j)=3,k=2). \label{eqn:1-z312}
\end{align}

But $3(\ref{eqn:1-z111})+(\ref{eqn:1-z123})+2(\ref{eqn:1-z222})+(\ref{eqn:1-z312})$ gives $0\ge \alpha_{1}$, a contradiction! So $g \notin \AM(\mathcal{T}^1,\mathcal{T}^2)$.
This means $g \notin \IC(M(\alpha_{1}\mathcal{T}^{1},\alpha_{2}\mathcal{T}^{2},\gamma))$ for any positive constants $\alpha_1,\alpha_2 > 0$ and $\gamma\in \R^{3}$. Figure \ref{fig:T1+T2} (left) gives a visual proof that $g \notin \IC(M(\mathcal{T}^1,\mathcal{T}^2,\mathbf{0}))$. 

By Theorem \ref{thm:main}, there exists another pair of type spaces $\mathcal{S}^1,\mathcal{S}^2$ such that $\IC(\mathcal{S}^1) = \IC(\mathcal{T}^1)$, and $g\in \AM(\mathcal{S}^{1},\mathcal{S}^{2})$. In this example, we can take $\mathcal{S}^{1}=\mathcal{T}^{1}$, and let $\mathcal{S}^2$ be the type space with type matrix
	\begin{equation*}
		S^{2}=\begin{bmatrix}
		0& 5& 4.3 \\
		0& -2& 0.5 \\
		0& 1& 3
		\end{bmatrix}.
	\end{equation*}
For convenience, we let $\calC$ be the type space $M(\mathcal{T}^1 ,\mathcal{T}^2,\mathbf{0})$ and $\calC'$ be the type space $ M(\mathcal{T}^1 ,\mathcal{S}^2,\mathbf{0})$. Then the type matrix $C$ ($C'$) of $\calC$ ($\calC'$) is a $9\times 3$ matrix whose $(i,j)$-th row is the sum of the $i$-th row of $T^{1}$ and the $j$-th row of $T^{2}$ ($S^{2}$). Figure \ref{fig:T1+T2} (right) shows that $g\in \IC(\calC')$ and thus $g\in \AM(\calS^{1},\calS^{2})$. 

	\begin{figure}[ht!]
		\begin{minipage}[t]{0.48\textwidth}
			\vspace{0pt}
			\centering
			\begin{tikzpicture}[scale=0.5]
				\filldraw[black!15, fill opacity = 0.7] (0,12) -- (0,17) -- (10,17) -- (10,12) -- (0,12) ; 
				
				\filldraw[black!15, fill opacity = 0.7] (8,2) -- (10,2) -- (10,-4) -- (2,-4) -- (8,2);
				
				\filldraw[gray] (7,-2) circle (2.5pt);
				\draw [gray,line width=0.1mm] (5,-4) -- (7,-2) -- (7,17);
				\draw [gray,line width=0.1mm] (7,-2) -- (10,-2);
				\node[below] at (7,-2) {\color{gray}$C_{1,1}$};
				
				\filldraw (0,12) circle (2.5pt);
				\draw [line width=0.1mm] (-2,10) -- (0,12) -- (0,17);
				\draw [line width=0.1mm] (0,12) -- (10,12);
				\node[below] at (0,12) {$C_{1,2}$};
				
				\filldraw[gray] (3,6) circle (2.5pt);
				\draw [gray,line width=0.1mm] (-2,1) -- (3,6) -- (3,17);
				\draw [gray,line width=0.1mm] (3,6) -- (10,6);
				\node[below] at (3,6) {\color{gray}$C_{1,3}$};
				
				\filldraw[gray] (9,-3) circle (2.5pt);
				\draw [gray,line width=0.1mm] (8,-4) -- (9,-3) -- (9,17);
				\draw [gray,line width=0.1mm] (9,-3) -- (10,-3);
				\node[below] at (9,-3) {\color{gray}$C_{2,1}$};
				
				\filldraw[gray] (2,11) circle (2.5pt);
				\draw [gray,line width=0.1mm] (-2,7) -- (2,11) -- (2,17);
				\draw [gray,line width=0.1mm] (2,11) -- (10,11);
				\node[below] at (2,11) {\color{gray}$C_{2,2}$};
				
				\filldraw[gray] (5,5) circle (2.5pt);
				\draw [gray,line width=0.1mm] (-2,-2) -- (5,5) -- (5,17);
				\draw [gray,line width=0.1mm] (5,5) -- (10,5);
				\node[below] at (5,5) { \color{gray} $C_{2,3}$};
				
				\filldraw (8,2) circle (2.5pt);
				\draw [line width=0.1mm] (2,-4) -- (8,2) -- (8,17);
				\draw [line width=0.1mm] (8,2) -- (10,2);
				\node[below] at (8,2) {$C_{3,1}$};
				
				\filldraw[gray] (1,16) circle (2.5pt);
				\draw [gray, line width=0.1mm] (-2,13) -- (1,16) -- (1,17);
				\draw [gray, line width=0.1mm] (1,16) -- (10,16);
				\node[below] at (1,16) {\color{gray} $C_{3,2}$};
				
				\filldraw[gray] (4,10) circle (2.5pt);
				\draw [gray, line width=0.1mm] (-2,4) -- (4,10) -- (4,17);
				\draw [gray, line width=0.1mm] (4,10) -- (10,10);
				\node[below] at (4,10) {\color{gray} $C_{3,3}$};
				
				\draw[line width=0.1mm] (-2,-4) rectangle (10,17);
			\end{tikzpicture}
		\end{minipage}
		\begin{minipage}[t]{0.48\textwidth}
			\vspace{0pt}
			\centering
			\begin{tikzpicture}[scale=0.5]
				\filldraw[black!15, fill opacity = 0.7] (7,3.3) -- (7,3.5) -- (5,3.5) -- (5,2) -- (5.7,2) -- (7,3.3);
				
				\filldraw[gray] (7,3.3) circle (2.5pt);
				\draw [line width=0.1mm] (-0.3,-4) -- (7,3.3) -- (7,17);
				\draw [line width=0.1mm] (7,3.3) -- (10,3.3);
				\node[above right] at (7,3.3) {\color{gray} $C'_{1,1}$};
				
				\filldraw[gray] (0,-0.5) circle (2.5pt);
				\draw [line width=0.1mm] (-2,-2.5) -- (0,-0.5) -- (0,17);
				\draw [line width=0.1mm] (0,-0.5) -- (10,-0.5);
				\node[left] at (0,-0.5) {\color{gray} $C'_{1,2}$};
				
				\filldraw[gray] (3,2) circle (2.5pt);
				\draw [line width=0.1mm] (-2,-3) -- (3,2) -- (3,17);
				\draw [line width=0.1mm] (3,2) -- (10,2);
				\node[below] at (3,2) {\color{gray} $C'_{1,3}$};
				
				\filldraw[gray] (9,2.3) circle (2.5pt);
				\draw [line width=0.1mm] (2.3,-4) -- (9,2.3) -- (9,17);
				\draw [line width=0.1mm] (9,2.3) -- (10,2.3);
				\node[below] at (9,2.3) {\color{gray} $C'_{2,1}$};
				
				\filldraw[gray] (2,-1.5) circle (2.5pt);
				\draw [line width=0.1mm] (-0.5,-4) -- (2,-1.5) -- (2,17);
				\draw [line width=0.1mm] (2,-1.5) -- (10,-1.5);
				\node[below right] at (2,-1.5) {\color{gray} $C'_{2,2}$};
				
				\filldraw[gray] (5,1) circle (2.5pt);
				\draw [line width=0.1mm] (0,-4) -- (5,1) -- (5,17);
				\draw [line width=0.1mm] (5,1) -- (10,1);
				\node[below right] at (5,1) {\color{gray} $C'_{2,3}$};
				
				\filldraw[gray] (8,7.3) circle (2.5pt);
				\draw [line width=0.1mm] (-2,-2.7) -- (8,7.3) -- (8,17);
				\draw [line width=0.1mm] (8,7.3) -- (10,7.3);
				\node[above right] at (8,7.3) {\color{gray} $C'_{3,1}$};
				
				\filldraw[gray] (1,3.5) circle (2.5pt);
				\draw [line width=0.1mm] (-2,0.5) -- (1,3.5) -- (1,17);
				\draw [line width=0.1mm] (1,3.5) -- (10,3.5);
				\node[above right] at (1,3.5) {\color{gray} $C'_{3,2}$};
				
				\filldraw[gray] (4,6) circle (2.5pt);
				\draw [line width=0.1mm] (-2,0) -- (4,6) -- (4,17);
				\draw [line width=0.1mm] (4,6) -- (10,6);
				\node[above right] at (4,6) {\color{gray} $C'_{3,3}$};
				
				\draw[line width=0.1mm] (-2,-4) rectangle (10,17);
			\end{tikzpicture}
		\end{minipage}
		\caption{The left figure is the hyperplane arrangement obtained from $\calC = M(\mathcal{T}^1,\mathcal{T}^2,\mathbf{0})$ for $\mathcal{T}^1,\mathcal{T}^2$ in Example \ref{exa:arrangement}.
		Here $C_{i,j} = t^1_i + t^2_j$ for $i,j = 1, \dots 3$. Note that the closed sectors $\minH_1(t^1_1 + t^2_2)$ and $\minH_3(t^1_3 + t^2_1)$ are disjoint, and thus no $h \in \IC(\calC)$ can simultaneously have $h(1,2)=1$ and $h(3,1)=3$, while $g(1,2)=1$ and $g(3,1)=3$. Therefore, $g \notin \IC(\calC)$. The right figure shows that $g\in \IC(\calC')$ as the shaded sector has cell type $g$ - for any point $p$ in the closed shaded sector and $1\le i,j\le 3$, $p\in \minH_{g(i,j)}(t^{1}_{i}+t^{2}_{j})$.}\label{fig:T1+T2}
	\end{figure}
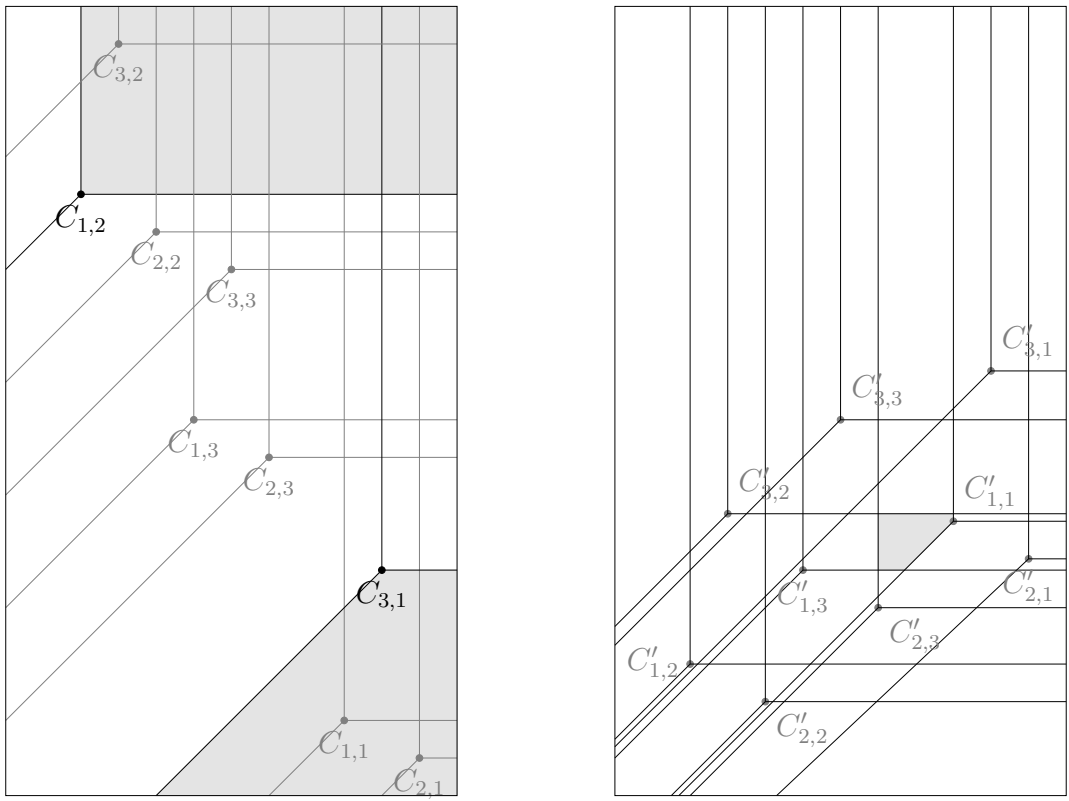
\end{example}

\section{Type spaces and incentive compatible outcome functions}\label{sec:typespaces}
We now state our main technical results, Proposition \ref{prop:minordet}, Theorem \ref{thm:equalIC}. In order, they solve the following problems: 
\begin{enumerate}
  \item Given an outcome function $g$ and a single-player type space $\T$, decide if $g \in \IC(\T)$.
  \item Given single-player type space $\T^1,\T^2$, decide if $\IC(\T^1) = \IC(T^2)$.
\end{enumerate}
The characterizations are given based on intrinsic properties of the tropical determinant of the corresponding type matrices. All of these results play important roles in the proof of Theorem \ref{thm:main}.

\subsection{Single-player case}

\begin{lemma}\label{lem:ic}
Let $g \in [m]^r$. Then $g \in \IC(\calT)$ for some type space $\calT$ with type matrix $T \in \R^{r \times m}$ if and only if the following system of linear inequalities in variables $x = (x_{k}) \in \R^{m}$ and parameters $T = (T_{i,k})\in \R^{r \times m}$ is feasible
\begin{equation}\label{eqn:typespace}
	T_{i,g(i)}-x_{g(i)} \geq T_{i,k}-x_{k} \quad \forall 1\le i\le r, 1\le k\le m, k \ne g(i).
\end{equation}
In particular, the set of inequalities above is equivalent to
	\begin{equation}\label{eqn:reducedtypespace}
		T_{i,g(i)}-x_{g(i)} \geq T_{i,k}-x_{k} \quad \forall 1\le i\le r, 1\le k\le m, k\notin \{g(j)\mid 1\le j\le r \}.
	\end{equation}
\end{lemma}

\begin{proof}
	Suppose $x = (x_{1},x_{2},\ldots,x_{m})$ is a feasible solution of (\ref{eqn:typespace}). We define a payment function $p: [m] \to \R$ by $p(j)=x_{j}$ for $1\le j\le m$. For any $t,s\in \calT$, there exists indices $1\le i,j\le r$ such that $\calT_{i}=t$ and $\calT_{j}=s$. Then (\ref{eqn:ic.2}) becomes
	
	\[T_{i,g(i)} - x_{g(i)} \geq T_{i,g(j)} - x_{g(j)}, \]
	
	which is exactly (\ref{eqn:typespace}) when taking $k=g(j)$.
	Conversely, if $g \in \IC(\calT)$, by definition there exists a payment function $p:[m]\to \R$ satisfying (\ref{eqn:ic.2}). Then $x_{j}=p(j)$ gives a solution to the system (\ref{eqn:typespace}) and thus it is feasible. Finally, note that for all $k$ not in the range of $g$, $x_{k}$ can only appear in the right hand side of (\ref{eqn:typespace}), thus after assigning values for all other $x$ variables, we can always assign sufficiently large values to those $x_{k}$. Hence we may delete all inequalities involving such $k$ to obtain the linear system \eqref{eqn:reducedtypespace}.
\end{proof}

We would like to characterize the feasibility of linear systems like (\ref{eqn:reducedtypespace}) only by properties of the entries of the matrix $T$. To make statements clear (for example Theorem \ref{thm:equalIC}), in the following definition we prefer to use the max-plus arithmetic instead of the min-plus arithmetic. 

\begin{definition}[Tropical determinant of a matrix]
The max-plus tropical determinant of a $k \times k$ matrix $T$ is the usual determinant with arithmetic carried out in the max-plus algebra, that is,
\begin{equation}\label{eqn:tdet} 
\tdet(T) = \overline{\bigoplus}_{\sigma \in \mathbb{S}_k} \bigodot_{j=1}^k T_{j,\sigma(j)} = \max_{\sigma\in\mathbb{S}_k} \sum_{j=1}^k T_{j,\sigma(j)},
\end{equation}
where $\mathbb{S}_k$ is the permutation group on $k$ letters.
\end{definition}

The following result was found by Joswig and Loho \cite{JoswigLoho}[Proposition 37].  For completeness and ease of reference, we restate it here. 
\begin{proposition}[\cite{JoswigLoho}]\label{prop:minordet}
	Let $\calT$ be a type space with type matrix $T \in \R^{r \times m}$ and $g\in [m]^{r}$. For any family $I$ of indices $1\le i_{1}<i_{2}<\cdots < i_{l}\le r$ such that $g(i_{1}), g(i_{2}), \cdots, g(i_{l})$ are distinct elements in $[m]$, there is an $l\times l$ minor $T_{I}$ of $T$ with row indices in $I$ and column indices in $J$, where $j_{z}=g(i_{z})$ for $1\le z\le l$ and $J=\{j_{1},j_{2},\cdots,j_{l}\}$. Then $g\in \IC(\calT)$ if and only if for all such $I$,  the tropical product
	\[\bigodot_{z=1}^{l} T_{i_{z},j_{z}} = \sum_{z=1}^{l}{T_{i_{z},j_{z}}} \]
	is the max-plus tropical determinant of $T_{I}$.
\end{proposition}

\begin{theorem}\label{thm:equalIC}
	Let $\calT^{1},\calT^{2}$ be type spaces with type matrix $T^{1},T^{2} \in \R^{r \times m}$, respectively. Then $\IC(\calT^{1})=\IC(\calT^{2})$ if and only if for every pair of index sets $I\subseteq [r]$ and $J\subseteq [m]$ with $|I|=|J|$, the sets of permutations that attain the max-plus tropical determinant are the same for the minors $T^{1}_{I,J}$ and $T^{2}_{I,J}$, where $T_{P,Q}$ denoted the minor of a matrix $T$ with row indices in $P$ and column indices in $Q$.
\end{theorem}

\begin{proof}
	Suppose $T^{1}_{I,J}$ and $T^{2}_{I,J}$ have the same set of permutations that attain the max-plus tropical determinant for all such $I$ and $J$. For any particular $g\in [m]^{r}$, the condition in Proposition \ref{prop:minordet} either holds for both $T^{1}$ and $T^{2}$, or does not hold for neither of $T^{1}$ and $T^{2}$. Hence by Proposition \ref{prop:minordet}, $g\in \IC(\calT^{1})$ if and only if $g\in \IC(\calT^{2})$, and thus $\IC(\calT^{1})=\IC(\calT^{2})$. The "if" part is done.
	
	Conversely, suppose $T^{1}_{I,J}$ and $T^{2}_{I,J}$ do not have the same set of permutations that attain the max-plus tropical determinant for all such $I$ and $J$, we need to show that $\IC(\calT^{1})\ne \IC(\calT^{2})$. There exists a pair of index sets $I=\{i_{1},\cdots,i_{l}\}$ and $J=\{j_{1},\cdots,j_{l}\}$ and a permutation on $J$ such that it corresponds to the max-plus tropical determinant in one of the minors of $T^{1}$ and $T^{2}$ but not in the other. We may assume that $\sum_{z=1}^{l}{T^{1}_{i_{z},j_{z}}}$ is a max-plus tropical determinant while $\sum_{z=1}^{l}{T^{2}_{i_{z},j_{z}}}$ is not. So there exists a permutation $\tau$ on $J$ such that
	
	\[\sum_{z=1}^{l}{T^{2}_{i_{z},\tau(j_{z})}} > \sum_{z=1}^{l}{T^{2}_{i_{z},j_{z}}}.\]
	
	By Proposition \ref{prop:minordet}, for any $g\in [m]^{r}$ such that $g(i_{z})=j_{z}$ for $1\le z\le l$, $g$ does not belong to $\IC(\calT^{2})$.
	It suffices to show that there exists such a $g$ that belongs to $\IC(\calT^{1})$. We consider $\IC(\calT^{1}_{I,J})$, where $\calT_{I,J}$ is a type space with type matrix $T^{1}_{I,J}\in \R^{l\times l}$. Labeling the $z$-th row by $i_{z}$ and the $z$-th column by $j_{z}$, the the identity permutation always attains the max-plus tropical determinant for all principal minors of $T^{1}_{I,J}$ (otherwise we can find a larger value for $tdet\left(T^{1}_{I,J}\right)$, a contradiction). So $T^{1}_{I,J}$ satisfies the conditions in Proposition \ref{prop:minordet} for $g|_{I}$. As a result, there exists real numbers $x_{i_{z}}$ for $1\le z\le l$ such that (\ref{eqn:typespace}) holds when $i\in I$ and $k\in J$.
	
	Now it suffices to define $g(i)$ for $i\in [r]-I$ and construct $x_{j}$ for $j\in [m]-J$. For each $i\in [r]-I$, we let $g(i)$ be an index in $J \subseteq [m]$ such that $T_{i,g(i)}-x_{g(i)}$ attains the maximum of the set $\{\left(T_{i,j} - x_{j}\right)  \mid j\in J \}$. Then our $g$ has range $J$. As explained after the proof of Lemma \ref{lem:ic}, we just need to check that (\ref{eqn:reducedtypespace}) holds for all $i\in [r]-I$ and $k\in J$. Suppose $k=j_{z}$ for some $1\le z\le l$, then (\ref{eqn:reducedtypespace}) becomes
	
	\[T_{i, g(i)} - x_{g(i)} \ge T_{i, j_{z}} - x_{j_{z}}.\]
	
	This inequality holds by the choice of $g(i)$, and we showed the feasibility of the system (\ref{eqn:reducedtypespace}). Thus our $g\in \IC(\calT^{1}) - \IC(\calT^{2})$, the "only if" part is done.
\end{proof}

\begin{remark}
Fink and Rincon \cite{fink2015stiefel} defined the matching multifield of a matrix, which keeps track of the set of permutations that achieve the tropical determinant for all maximal minors. They showed \cite[Theorem 4.4]{fink2015stiefel} that if $\IC(\calT^{1})=\IC(\calT^{2})$, then the matching multifields of $T_{1}$ and $T_{2}$ are equal. Theorem \ref{thm:equalIC} strengthens this result by showing that the converse holds when one requires equality not just for maximal minors, but for all minors. 
\end{remark}

Given a subset $V\subseteq [m]^{r}$, we would like to find a characterization of all type spaces $\calT$ with type matrix $T\in \R^{r\times m}$ such that $\IC(\calT)=V$. For every $g\in V$, we define a family of non-strict inequalities with variables $x^{g}_{i}(1\le i\le m)$ and parameters $T_{i,j}(1\le i\le r, 1\le j\le m)$
	
\begin{equation}\label{eqn:goodnonstrict}
	T_{i,g(i)} - x^{g}_{g(i)} \ge T_{i,k} - x^{g}_{k} \quad \forall 1\le i\le r, 1\le k\le m, k\ne g(i).
\end{equation}

\subsection{Multi-player case}
For convenience, we also use a $r_{1}\times r_{2}\times \cdots \times r_{n}$ tensor to denote $g$ - we let its $(i_{1},\ldots,i_{n})$-th entry be the value $g\left(\calT^{1}_{i_{1}},\ldots, \calT^{n}_{i_{n}}\right)$. For $n\ge 2$, \cite[Proposition 6.1]{CrowellTran1} gave a simple characterization for a given tensor to be in $\IC(\calT^{1},\ldots,\calT^{n})$. For self-containment we recall their results here. 

\begin{proposition}\label{prop:genIC}
	Let $\calT^{1},\ldots,\calT^{n}$ be type spaces with type matrices $T^{1}\in \R^{r_{1}\times m}, \ldots, T^{n}\in \R^{r_{n}\times m}$. For any $g\in [m]^{r_{1}\times r_{2}\times \cdots \times r_{n}}$, $g\in \IC(\calT^{1},\ldots,\calT^{n})$ if and only if for every $1\le j\le n$ and any $1\le i_{k}\le r_{k}$ for $k\ne j$, the function $h: [r_{j}] \to [m]$ with $h(l)=g(i_{1},i_{2},\ldots,i_{j-1},l,i_{j+1},\ldots,i_{n})$ for $1\le l\le r_{j}$ belongs to $\IC(\calT^{j})$. 
\end{proposition}

In particular, when $n=2$ we have the following corollary. 

\begin{corollary}\label{cor:IC2}
	For any matrix $g\in [m]^{r_{1}\times r_{2}}$ representing a function $[r_{1}]\times [r_{2}] \to [m]$  and two type spaces $\calT^{1},\calT^{2}$, $g\in \IC(\calT^{1},\calT^{2})$ if and only every column vector of $g$ belongs to $\calT^{1}$ and every row vector of $g$ belongs to $\calT^{2}$.
\end{corollary}

By replacing the entries $T_{i,j}$ using (\ref{eqn:tensorcomb}) in Lemma \ref{lem:ic}, we have the following characterization for tensors which are affine maximizers. 

\begin{lemma}\label{lem:am}
	Let $g \in [m]^{r_{1} \times \cdots \times r_{n}}$. There exists $U^{1} \in \R^{r_{1} \times m}, \ldots, U^{n} \in \R^{r_{n} \times m}$ such that $g \in \AM(U^{1},\ldots,U^{n})$ if and only if the following system in the variables $Z, U^{1}, \ldots, U^{n}$ is feasible:
	
	\begin{equation}\label{eqn:Z.AM}
			\sum_{j=1}^{n}{U^{j}_{i_{j}, g(i_{1},\ldots,i_{n})}} - Z_{g(i_{1},\ldots,i_{n})}
			\ge \sum_{j=1}^{n}{U^{j}_{i_{j}, k}} - Z_{k} \quad \forall 1\le i_{j}\le r_{j}, 1\le k\le m, k\ne g(i_{1},\ldots,i_{n}). 
	\end{equation}
\end{lemma}

\section{Proofs}\label{sec:proof}

The proof of Theorem \ref{thm:main} relies on Farkas Lemma, Lemma \ref{lem:ic} and \ref{lem:am} that we introduced in Section \ref{sec:typespaces}. Analyzing the sum of inequalities in cycles, we argue that the nonexistence of such a pair $\calS^{1},\calS^{2}$ would imply the nonexistence of $\calT^{1}$, which is a contradiction to our assumption. The proofs of Theorem \ref{thm:main.counter} and \ref{thm:2p-sym.counter} involve exhibiting specific type spaces where the claimed statements fail.

\subsection{Proof of Theorem \ref{thm:main}}

Fix type space $\calT^{1}$ with type matrix $T^{1} \in \R^{r_1 \times m}$ and type space $\calT^{2}$ with type matrix $T^{2} \in \R^{r_2 \times m}$ and $g\in [m]^{r_{1}\times r_{2}}$ such that $g\in \IC(\calT^{1}, \calT^{2})$. To prove Theorem \ref{thm:main}, we need to show that there exist type spaces $\calS^{2}$ with type matrix $S^{2} \in \R^{r_2 \times m}$ such that $g\in \AM(\calT^{1}, \calS^{2})$. By Lemma \ref{lem:am} (when $n=2$), it suffices to show that there exist real numbers $S^{2}_{j,k}(1\le j\le r_{2}, 1\le k\le m)$ and $Z_{k}(1\le k\le m)$ such that for $1\le i\le r_{1}, 1\le j\le r_{2}, 1\le k\le m$, we have 

\begin{equation}\label{eqn:main-AM}
	T^{1}_{i,g(i,j)}+S^{2}_{j,g(i,j)} - Z_{g(i,j)} \ge T^{1}_{i,k}+S^{2}_{j,k} - Z_{k}.
\end{equation}

For $1\le j\le r_{2}$, let $g^{j}$ be the $j$-th column vector of $g$. By Corollary \ref{cor:IC2}, since $g\in \IC(\calT^{1},\calT^{2})$, every $g^{j}\in \IC(\calT^{1})$. By (\ref{eqn:goodnonstrict}), there exist corresponding vectors $x^{j}\in \R^{m}$ such that for $1\le i\le r_{1}, 1\le j\le r_{2}, 1\le k\le m$, we have (note that the $i$-th entry of $g^{j}$ is just $g(i,j)$)

\begin{equation}\label{eqn:V}
	T^{1}_{i,g(i,j)} - x^{j}_{g(i,j)} \ge T^{1}_{i,k} - x^{j}_{k}.
\end{equation}

Now for $1\le j\le r_{2}, 1\le k\le m$, we let $S^{2}_{j,k} = -x^{j}_{k}$ and $Z_{k}=0$. Then $S^{2}_{j,g(i,j)} = - x^{j}_{g(i,j)}$, and (\ref{eqn:main-AM}) becomes

\[ T^{1}_{i,g(i,j)} - x^{j}_{g(i,j)} \ge T^{1}_{i,k} - x^{j}_{k}, \]

which is exactly (\ref{eqn:V}). So our choice of $S^{2}$ and $Z$ satisfy (\ref{eqn:main-AM}), and Theorem \ref{thm:main} is proved.

\subsection{Proof of Theorem \ref{thm:main.counter}(i)}

To prove Theorem \ref{thm:main.counter}(i), we shall construct an explicit type space $\calT = \calT^1 \times \calT^2 \times \calT^3$ on three players and an IC mechanism $g$ on $\calT$ such that for any type space $\calS = \calS^1 \times \calS^2 \times \calS^3$,
$g \in \AM(\calS^{1},\calS^{2},\calS^{3})$ and $\IC(\calT^{1})=\IC(\calS^{1})$ cannot happen simultaneously. Let $r_{1}=r_{2}=r_{3}=2$ and $m=6$. Define $g \in \{1,2,\dots,6\}^{2 \times 2 \times 2}$ as
	\begin{align*}
		&g(1,1,1)=1, g(1,1,2)=3, g(1,2,1)=4, g(1,2,2)=1, \\
		&g(2,1,1)=5, g(2,1,2)=2, g(2,2,1)=2, g(2,2,2)=6.
	\end{align*}
Define $\calT^{1},\calT^{2},\calT^{3}$ to be the set of the column vectors of the following matrices 
	\begin{equation*}
		T^{1} = \begin{bmatrix}
			3& 4& 5& 6& 2& 1 \\
		 	0& 0& 0& 0& 0& 0
		\end{bmatrix}, \quad
		T^{2} = \begin{bmatrix}
		5& 2& 3& 6& 4& 1 \\
		0& 0& 0& 0& 0& 0
	\end{bmatrix},\quad
		T^{3} = \begin{bmatrix}
		5& 2& 6& 3& 4& 1 \\
		0& 0& 0& 0& 0& 0
	\end{bmatrix}.
	\end{equation*}
	
We verify that $g\in \IC(\calT^{1},\calT^{2},\calT^{3})$ by checking that the column vectors of corresponding flattenings of $g$ belong to $\IC(\calT^{i})$, where $i=1,2,3$.
	\begin{equation*}
		\begin{split}
			(1,5), (3,2), (4,2), (1,6) &\in \IC(\calT^{1}); \\
			(1,3), (4,1), (5,2), (2,6) &\in \IC(\calT^{2}); \\
			(1,4), (3,1), (5,2), (2,6) &\in \IC(\calT^{3}).
		\end{split}
	\end{equation*}
	
	Now, suppose $g \in AM(\calS^1,\calS^2,\calS^3)$ for some type space $\calS = \calS^1 \times \calS^2 \times \calS^3$, each $\calS^{i}$ represented by the $2\times 2$ matrix $S^{i}$. By Lemma \ref{lem:am}, this holds if and only if the following system of linear inequalities in variable $z$ is feasible:
\vskip-0.5cm	
	\begin{equation*}
S^{1}_{i_{1},g(i_{1},i_{2},i_{3})}+S^{2}_{i_{1},g(i_{1},i_{2},i_{3})}+S^{3}_{i_{1},g(i_{1},i_{2},i_{3})} - z_{g(i_{1},i_{2},i_{3})} \ge S^{1}_{i_{1},k}+S^{2}_{i_{2},k}+S^{3}_{i_{3},k}-z_{k}, 
	\end{equation*}
where the indices run over all $i_{1},i_{2},i_{3},k$ such that $1\le i_{1},i_{2},i_{3}\le 2, 1\le k\le 6, k\ne g(i_{1},i_{2},i_{3}).$ In particular, we have
	
	\begin{equation}\label{ine:333circuit}
	\begin{split}
	S^{1}_{1,1}+S^{2}_{1,1}+S^{3}_{1,1}-z_{1} &\ge S^{1}_{1,2}+S^{2}_{1,2}+S^{3}_{1,2}-z_{2} \quad [(i_1,i_2,i_3,k)=(1,1,1,2)]; \\
	S^{1}_{1,1}+S^{2}_{2,1}+S^{3}_{2,1}-z_{1} &\ge S^{1}_{1,2}+S^{2}_{2,2}+S^{3}_{2,2}-z_{2} \quad [(i_1,i_2,i_3,k)=(1,2,2,2)]; \\
	S^{1}_{2,2}+S^{2}_{1,2}+S^{3}_{2,2}-z_{2} &\ge S^{1}_{2,1}+S^{2}_{1,1}+S^{3}_{2,1}-z_{1} \quad [(i_1,i_2,i_3,k)=(2,1,2,1)]; \\
	S^{1}_{2,2}+S^{2}_{2,2}+S^{3}_{1,2}-z_{2} &\ge S^{1}_{2,1}+S^{2}_{2,1}+S^{3}_{1,1}-z_{1} \quad [(i_1,i_2,i_3,k)=(2,2,1,1)].
	\end{split}
	\end{equation}
	
	Summing over the four inequalities in (\ref{ine:333circuit}), we get 
	
	\begin{equation*}
		S^{1}_{1,1}+S^{1}_{2,2} \ge S^{1}_{1,2}+S^{1}_{2,1}.
	\end{equation*} 

Now suppose that in addition, $\IC(\calS^{1})=\IC(\calT^{1})$. Since the max-plus tropical determinant of $T^{1}_{\{1,2\}, \{1,2\}}$ is only attained by $T^{1}_{1,1}+T^{1}_{2,2}$, by Theorem \ref{thm:equalIC} we also have

\begin{equation}\label{eqn:s.strict.more}
		S^{1}_{1,1}+S^{1}_{2,2} > S^{1}_{1,2}+S^{1}_{2,1}.
\end{equation}

	However, $(2,1)\in \IC(\calT^{1})=\IC(\calS^{1})$, which implies that there exist $X_{1},X_{2}\in \R$ such that
	\begin{equation*}
			S^{1}_{1,2} - X_{2} \ge S^{1}_{1,1} - X_{1}; \quad 
			S^{1}_{2,1} - X_{1} \ge S^{1}_{2,2} - X_{2}.
	\end{equation*}
Summing up these two inequalities to eliminate $X_{1},X_{2}$, we obtain	
	\begin{equation*}
	S^{1}_{1,2}+S^{1}_{2,1} \ge S^{1}_{1,1}+S^{1}_{2,2},
	\end{equation*} 
but this contradicts (\ref{eqn:s.strict.more}). Hence $g\in \AM(\calS^{1},\calS^{2},\calS^{3})$ and $\IC(\calT^{1})=\IC(\calS^{1})$ cannot happen simultaneously.

\subsection{Proof of Theorem \ref{thm:main.counter}(ii)}

To prove Theorem \ref{thm:main.counter}(ii), we construct an explicit type space $\calT = \calT^1 \times \calT^2$ and an IC mechanism $g$ on $\calT$ such that for any type space $\calS = \calS^1 \times \calS^2$, $g \in AM(\calS^1,\calS^2)$, $\IC(\calS^1) = \IC(\calT^1)$ and $\IC(\calS^2) = \IC(\calT^2)$ cannot happen simultaneously. Let $r_{1}=r_{2}=2$ and $m=4$. Define
	\begin{equation*}
	T^{1}=\begin{bmatrix}
	13& 46& 9& 11 \\
	45& 47& 1& 24
	\end{bmatrix},
	T^{2}=\begin{bmatrix}
	12& 8 & 19& 38 \\
	28& 46& 19& 4
	\end{bmatrix},
	g=\begin{bmatrix}
	4 & 3 \\ 1 & 2
	\end{bmatrix}.
	\end{equation*}
	And we let $\calT^{i}$ be the type spaces with type matrix $T^{i}$ for $i=1,2$. We have that
	\begin{align*}
	\IC(\calT^{1})&=\{(1, 1), (2, 1), (2, 2), (2, 4), (3, 1), (3, 2), (3, 3), (3, 4), (4, 1), (4, 4) \} \\
	\IC(\calT^{2})&=\{(1, 1), (1, 2), (2, 2), (3, 1), (3, 2), (3, 3), (4, 1), (4, 2), (4, 3), (4, 4) \},
	\end{align*}
so $g \in \IC(\calT^{1},\calT^{2})$. Now, suppose for contradiction that there exists a type space $\calS = \calS^1 \times \calS^2$ that satisfy all three conditions. Since $\IC(\calS^1) = \IC(\calT^1)$ and $T^{1}_{1,2}+T^{1}_{2,4} = 46+24 =70 > 58 = 11+47 = T^{1}_{1,4}+T^{1}_{2,2}$, by Theorem \ref{thm:equalIC} we also have

\begin{equation}\label{ine:224-s1}
	S^{1}_{1,2} + S^{1}_{2,4} > S^{1}_{1,4} + S^{1}_{2,2}.
\end{equation} 

Since $\IC(\calS^2) = \IC(\calT^2)$ and $T^{2}_{1,3}+T^{2}_{2,1} = 19 + 28  > 12 + 19 = T^{2}_{1,1}+T^{2}_{2,3}$, by Theorem \ref{thm:equalIC} we also have

\begin{equation}\label{ine:224-s2}
	S^{2}_{1,3} + S^{2}_{2,1} > S^{2}_{1,1} + S^{2}_{2,3}.
\end{equation} 

Note that $\IC(\calS^1) = \IC(\alpha \calS^1)$ for all positive $\alpha > 0$, one may assume that the constants $\alpha_1, \alpha_2$ in the affine maximizer equal to one. Since $g\in \AM(\calS^{1},\calS^{2})$ with constants $\alpha_1 = \alpha_2 = 1$, by Lemma \ref{lem:am} we can find $Z\in \R^{4}$ such that
	\begin{equation}\label{ine:224-z}
	\begin{split}
	S^{1}_{1,4}+S^{2}_{1,4}-Z_{4}&\ge S^{1}_{1,3}+S^{2}_{1,3}-Z_{3}\\
	S^{1}_{1,3}+S^{2}_{2,3}-Z_{3}&\ge S^{1}_{1,2}+S^{2}_{2,2}-Z_{2}\\
	S^{1}_{2,1}+S^{2}_{1,1}-Z_{1}&\ge S^{1}_{2,4}+S^{2}_{1,4}-Z_{4}\\
	S^{1}_{2,2}+S^{2}_{2,2}-Z_{2}&\ge S^{1}_{2,1}+S^{2}_{2,1}-Z_{1}\\
	\end{split}
	\end{equation}
	
	Summing over all inequalities (\ref{ine:224-s1}), (\ref{ine:224-s2}) and (\ref{ine:224-z}), we have $0 > 0$, a contradiction.

\subsection{Proof of Theorem \ref{thm:main.counter}(iii)}

Let $m=3,n=2$ and $r_{1}=r_{2}=2$. Define
	
	\begin{equation*}
		T^{1}=\begin{bmatrix}
			0& 1& 3 \\
			0& 2& 1
		\end{bmatrix},
		T^{2}=\begin{bmatrix}
			0& 4& 2 \\
			0& 2& 0
		\end{bmatrix}, g^{1}=\begin{bmatrix}
		2& 3 \\
		2& 1
		\end{bmatrix},
		g^{2}=\begin{bmatrix}
		3& 1 \\
		3& 2
		\end{bmatrix}.
	\end{equation*}
	And let $\calT^{1},\calT^{2}$ be the type spaces with type matrices $T^{1},T^{2}$ respectively. We have
	
	\begin{equation*}
		\begin{split}
			\IC(\calT^{1})&=\{(1, 1), (1, 2), (2, 2), (3, 1), (3, 2), (3, 3) \} \\
			\IC(\calT^{2})&=\{(1, 1), (2, 1), (2, 2), (2, 3), (3, 1), (3, 2), (3, 3) \}. 
		\end{split}
	\end{equation*}
	
	By Corollary \ref{cor:IC2}, $g^{1},g^{2}\in \IC(\calT^{1},\calT^{2})$. To prove Theorem \ref{thm:main.counter}(iii), we need to show that there do not exist $\calS^{1},\calS^{2}$ with type matrices $S^{1},S^{2} \in \R^{2\times 3}$ such that
$g^{1},g^{2}\in \AM(\calS^{1},\calS^{2})$ and $\IC(\calS^{1})=\IC(\calT^{1})$. Suppose for contradiction that such $\calS^{1},\calS^{2}$ exist. By definition, there exist positive constants $\alpha_{1}, \alpha_{2}, \beta_{1}, \beta_{2}$ such that $g^{1}\in \IC(\alpha_{1}\calS^{1},\alpha_{2}\calS^{2})$ and $g^{2}\in \IC(\beta_{1}\calS^{1},\beta_{2}\calS^{2})$. Then there exist $Y,Z\in \R^{3}$ such that for $1\le i,j\le 2, 1\le p\le 3$, 
	
	\begin{equation*}
		\begin{split}
			\alpha_{1}S^{1}_{i,g^{1}(i,j)} + \alpha_{2}S^{2}_{j,g^{1}(i,j)} - Y_{g^{1}(i,j)} &\ge \alpha_{1}S^{1}_{i,p} + \alpha_{2}S^{2}_{j,p} - Y_{p}; \\
			\beta_{1}S^{1}_{i,g^{2}(i,j)}+\beta_{2}S^{2}_{j,g^{2}(i,j)}-Z_{g^{2}(i,j)} &\ge \beta_{1}S^{1}_{i,p}+\beta_{2}S^{2}_{j,p} - Z_{p}.
		\end{split}
	\end{equation*} 

	Hence we have the following inequalities:
	
	\begin{equation}\label{eqn:y3}
		\begin{split}
			\alpha_{1}S^{1}_{1,2}+\alpha_{2}S^{2}_{1,2} - Y_{2}&\ge \alpha_{1}S^{1}_{1,3}+\alpha_{2}S^{2}_{1,3}-Y_{3} \quad  [(i,j,g^{1}(i,j),p)=(1,1,2,3)]; \\
			\alpha_{1}S^{1}_{1,3}+\alpha_{2}S^{2}_{2,3} - Y_{3}&\ge \alpha_{1}S^{1}_{1,1}+\alpha_{2}S^{2}_{2,1}-Y_{1} \quad  [(i,j,g^{1}(i,j),p)=(1,2,3,1)]; \\
			\alpha_{1}S^{1}_{2,1}+\alpha_{2}S^{2}_{2,1} - Y_{1}&\ge \alpha_{1}S^{1}_{2,2}+\alpha_{2}S^{2}_{2,2}-Y_{2} \quad  [(i,j,g^{2}(i,j),p)=(2,2,1,2)]
		\end{split}
	\end{equation}
	
	and
	
	\begin{align}
			\beta_{1}S^{1}_{1,1}+\beta_{2}S^{2}_{2,1}- Z_{1}&\ge \beta_{1}S^{1}_{1,3}+\beta_{2}S^{2}_{2,3}-Z_{3}\quad [(i,j,g^{2}(i,j),p)=(1,2,1,3)]; \label{eqn:Z123} \\
			\beta_{1}S^{1}_{2,3}+\beta_{2}S^{2}_{1,3}- Z_{3}&\ge \beta_{1}S^{1}_{2,2}+\beta_{2}S^{2}_{1,2}-Z_{2}\quad [(i,j,g^{2}(i,j),p)=(2,1,3,2)]; \label{eqn:Z212} \\
			\beta_{1}S^{1}_{2,2}+\beta_{2}S^{2}_{2,2}- Z_{2}&\ge \beta_{1}S^{1}_{2,1}+\beta_{2}S^{2}_{2,1}-Z_{1}\quad [(i,j,g^{2}(i,j),p)=(2,2,2,1)]; \label{eqn:Z221} \\
			\beta_{1}S^{1}_{2,2}+\beta_{2}S^{2}_{2,2}- Z_{2}&\ge \beta_{1}S^{1}_{2,3}+\beta_{2}S^{2}_{2,3}-Z_{3}\quad [(i,j,g^{2}(i,j),p)=(2,2,2,3)]. \label{eqn:Z223}
	\end{align}
	
	Summing over the inequalities in (\ref{eqn:y3}), we get
	
	\begin{equation}\label{eqn:yz1}
		\alpha_{1}\cdot \left[S^{1}_{1,2} + S^{1}_{2,1} - S^{1}_{1,1} - S^{1}_{2,2}\right]+\alpha_{2}\cdot \left[S^{2}_{1,2} + S^{2}_{2,3} - S^{2}_{1,3} - S^{2}_{2,1} \right] \ge 0.
	\end{equation}
	
	Add up equations (\ref{eqn:Z123}),(\ref{eqn:Z212}) and (\ref{eqn:Z221}), we get 
	
	\begin{equation}\label{eqn:yz2}
		\beta_{1}\cdot \left[S^{1}_{1,1} + S^{1}_{2,3} - S^{1}_{1,3} - S^{1}_{2,1} \right] + \beta_{2}\cdot \left[ S^{2}_{1,3} + S^{2}_{2,2} - S^{2}_{1,2} - S^{2}_{2,3}\right] \ge 0.
	\end{equation}
	
	Add up (\ref{eqn:Z212}) and (\ref{eqn:Z223}), divide by $\beta_{2}>0$, we get
	
	\begin{equation}\label{eqn:s232}
		S^{2}_{1,3} + S^{2}_{2,2} - S^{2}_{1,2} - S^{2}_{2,3} \ge 0.
	\end{equation}
	
	Note that $\IC(\calS^{1})=\IC(\calT^{1})$ and 
	
	\begin{align}
		T^{1}_{1,1}+T^{1}_{2,2} = 0 + 2 &> 1 + 0 = T^{1}_{1,2} - T^{1}_{2,1}; \\
		T^{1}_{1,3}+T^{1}_{2,2} = 3 + 2 &> 1 + 1 = T^{1}_{1,2} - T^{1}_{2,3},
	\end{align}
	by Theorem \ref{thm:equalIC}, we also have

	\begin{align}
		S^{1}_{1,1}+S^{1}_{2,2} - S^{1}_{1,2} - S^{1}_{2,1} & > 0, \label{eqn:s112} \\
		S^{1}_{1,3}+S^{1}_{2,2} - S^{1}_{1,2} - S^{1}_{2,3} & > 0. \label{eqn:s123}
	\end{align}
	
	Now if $\alpha_{1}\beta_{2}\ge \alpha_{2}\beta_{1}$, then we take the sum
	
	\begin{equation*}
			\alpha_{2}\beta_{1}\cdot (\ref{eqn:s123}) + \left(\alpha_{1}\beta_{2} - \alpha_{2}\beta_{1} \right)\cdot (\ref{eqn:s112})
			+ \beta_{2}\cdot (\ref{eqn:yz1}) + \alpha_{2}\cdot (\ref{eqn:yz2}),
	\end{equation*}
	which is a $\R_{\ge 0}$-linear combination of these inequalities, and its LHS vanishes. Since (\ref{eqn:s123}) is a strict inequality and $\alpha_{2}\beta_{1}>0$, it implies $0>0$, a contradiction!
	
	Similarly, if $\alpha_{1}\beta_{2}<\alpha_{2}\beta_{1}$, then we take the sum
	\begin{equation*}
			\alpha_{1}\beta_{1}\cdot (\ref{eqn:s123}) + \left(\alpha_{2}\beta_{1} - \alpha_{1}\beta_{2} \right)\cdot (\ref{eqn:s232})
			+ \beta_{1}\cdot (\ref{eqn:yz1}) + \alpha_{1}\cdot (\ref{eqn:yz2}),
	\end{equation*}
	which is a $\R_{+}$-linear combination of these inequalities, and its LHS vanishes too. Since (\ref{eqn:s123}) is a strict inequality and $\alpha_{1}\beta_{1}>0$, it implies $0 > 0$, a contradiction, too! So in either cases there is a contradiction, and such a pair of $\calS^{1},\calS^{2}$ does not exist.

\subsection{Proof of Theorem \ref{thm:2p-sym.counter}}

We consider	a type space $\calT = \calT^{1} \times \calT^{1}$, here the type matrix corresponding to $\calT^{1}$ is
	\begin{equation*}
	T^{1}=\begin{bmatrix}
	0 & 5 & 6 \\
	0 & 0 & 3 \\
	0 & 4 & 0
	\end{bmatrix}. 
	\end{equation*}
	In addition we let
	\[g=\begin{bmatrix}
		2 & 3 & 2 \\
		3 & 3 & 1 \\
		2 & 1 & 1
		\end{bmatrix}. \]
	Then 
	\begin{equation*}
		\begin{split}
				\IC(\calT^{1})=&\{(1, 1, 1), (2, 1, 1), (2, 1, 2), (2, 2, 2), (2, 3, 2), \\
				& (3, 1, 1), (3, 1, 2), (3, 3, 1), (3, 3, 2), (3, 3, 3)\},
		\end{split}
	\end{equation*}
	by Corollary \ref{cor:IC2}, $g \in \IC(\calT^{1},\calT^{1})$. However, there is no type space $\calS$ with $3\times 3$ type matrix $S$ such that $g\in \AM(\calS, \calS)$ and 
	$\IC(\calS)=\IC(\calT)$. Indeed, suppose for contradiction that such $\calS$ exists. One may assume that $\alpha_1 = \alpha_2 = 1$. In addition, $(3,1,2)\in \IC(\calT)=\IC(\calS)$ (this cell type $(3,1,2)$ corresponds to the unique shaded cell in the right picture of Figure \ref{fig:T1+T2}). So for any $Y\in \R^{3}$ in the interior of that cell, by (\ref{eqn:typespace}) we have
	
	\begin{equation*}
	\begin{split}
		S_{1,3}-Y_{3} &> S_{1,k} - Y_{k} \quad \forall k=1,2; \\
		S_{2,1}-Y_{1} &> S_{2,k} - Y_{k} \quad \forall k=2,3; \\
		S_{3,2}-Y_{2} &> S_{3,k} - Y_{k} \quad \forall k=1,3.
	\end{split}
	\end{equation*}
	
	Since $g \in \AM(\calS,\calS)$ with constants $\alpha_1 = \alpha_2 = 1$, by Lemma \ref{lem:am} there exists $Z\in \R^{3}$ such that
	
	\begin{equation*}
	S_{i,g(i,j)} + S_{j,g(i,j)} - Z_{g(i,j)} \ge S_{i,k} + S_{j,k} - Z_{k} \quad \forall 1\le i,j,k\le 3, k\ne g(i,j).
	\end{equation*}
	
	So we have the following inequalities
	
	\begin{equation}\label{ine:3-circuit}
		\begin{split}
			2(S_{1,3} - Y_{3}) &> 2(S_{1,2} - Y_{2}) \\
			2(S_{2,1} - Y_{1}) &> 2(S_{2,3} - Y_{3}) \\
			2(S_{3,2} - Y_{2}) &> 2(S_{3,1} - Y_{1}) \\
			2S_{1,2} - Z_{2} &\ge 2S_{1,3} - Z_{3} \quad [(i,j,k)=(1,1,3)] \\
			2S_{2,3} - Z_{3} &\ge 2S_{2,1} - Z_{1} \quad [(i,j,k)=(2,2,1)] \\
			2S_{3,1} - Z_{1} &\ge 2S_{3,2} - Z_{2} \quad [(i,j,k)=(3,3,2)] \\
		\end{split}
	\end{equation}
	
	Summing over the inequalities in (\ref{ine:3-circuit}), we obtain the desired contradiction $0>0$. 

\section{Summary}\label{sec:summary}

In this paper, we take a novel view of the classical Roberts' Theorem and ask if a given mechanism can be turned into an affine maximizer on an equivalent type space. We give an affirmative result in this direction, Theorem \ref{thm:main}, which applies to all finite type spaces on two-player games. Through a series of counterexamples, we show that our theorems are strongest possible in this general setup. 
Our proof technique utilizes insights from tropical geometry, and raises a number of questions of interest to both economists and tropical geometers. 

The second open question is whether our theorems hold for continuous type spaces. One reason to pursue this question is to obtain a novel alternative proof of the original Roberts' Theorem that is distinct from the analysis-based techniques in the current literature, and to make our results easier to compare with the others on this topic. While compact type spaces can be approximated by discretization, and a number of tropical geometric results still hold in the limit \cite{CrowellTran1}, the dimensions of the matrices involved become infinite. In particular, one then has an infinite system of inequalities and variables, and thus the techniques applied in our paper no longer hold, presenting a significant challenge. For type spaces such as $\R^n$, the set of IC outcome functions completely determines the type space. Thus, a starting point for this open question would be: for what set of type spaces do the set of IC outcome functions completely determines the type space, and does Theorem \ref{thm:main} hold in those cases?

\section*{Acknowledgments}
We thank Benny Moldovanu and Robert Crowell for interesting conversations. We also thank an anonymous referee for careful reading and suggestions that played an important role in simplifying the proof of Theorem 1.1. Bo Lin is supported by a Simons Postdoctoral Research Fellowship at the University of Texas at Austin. Ngoc Mai Tran is supported by the Bonn Research Fellowship at the Hausdorff Center for Mathematics.

\bibliographystyle{plain}

\end{document}